\documentclass[11pt,a4paper]{amsart}
\usepackage[T1]{fontenc}
\usepackage{textcomp}
\usepackage{lmodern}
\usepackage{a4wide}
\usepackage{enumitem}
\usepackage{mathrsfs}
\usepackage{amsmath, amssymb}
\usepackage{amsthm}

\theoremstyle{plain}
\newtheorem{theorem}{Theorem}[section]
\newtheorem{lemma}[theorem]{Lemma}
\newtheorem{corollary}[theorem]{Corollary}
\newtheorem{prop}[theorem]{Proposition}

\theoremstyle{definition}
\newtheorem{definition}[theorem]{Definition}
\newtheorem{remark}[theorem]{Remark}

\newtheorem*{lemma*}{Lemma}
\newtheorem*{theorem*}{Theorem}
\newtheorem*{proposition*}{Proposition}
\newtheorem*{corollary*}{Corollary}
\newtheorem*{definition*}{Definition}

\newcommand{\C}{\mathbb{C}}

\newcommand{\R}{\mathbb{R}}

\newcommand{\cfct}{\mathbf{c}}

\DeclareMathOperator{\G}{G}
\newcommand{\curlH}[1]{\mathcal{H}#1}
\usepackage{tikz}
\usetikzlibrary{arrows,shapes,trees}
\begin{document}
\title{Remarks on the inverse Cherednik--Opdam transform on the real line}
\author{Troels Roussau Johansen}
\address{Mathematisches Seminar, Chr.-Albrechts Universit{\"a}t zu Kiel, Ludewig--Meyn-Str. 4, D-24118 Kiel, Germany}

\email{johansen@math.uni-kiel.de}
\begin{abstract}
We  obtain Hausdorff--Young inequalities for the one-dimensional Cherednik--Opdam transform and its inverse, and we establish a real Paley--Wiener theorem for its inverse that generalizes an analogous result by N. B. Andersen for the Jacobi transform.
\end{abstract}

\maketitle
\section{Introduction}
The present paper is concerned with themes in classical harmonic analysis in the framework of the Cherednik--Opdam transform $\mathcal{H}$ on the real line. The starting point is the Hausdorff--Young inequality which, among other things, allows an extension of $\mathcal{H}$ from $L^2(\R,d\mu)$ to $L^p(\R,d\mu)$ for $1\leq p\leq 2$; the extension is denoted $\mathscr{H}_p$. Similar statements can be established for the inverse transform $\mathcal{I}=\mathcal{H}^{-1}$, which can also be extended from $L^2(\R,d\nu)$ to a map $\mathscr{I}_q$ from $L^q(\R,d\nu)$ into $L^p(\R,d\mu)$. Our first result states that the inverse of the extended map $\mathscr{H}_p$ coincides $\nu$-almost everywhere with $\mathscr{I}_{p^\prime}$. This result was recently established for commutative hypergroups in \cite{HY-hyper} but the `Cherednik--Opdam convolution' on $\R$ does not give rise to a hypergroup structure so the result does not follow immediately. The strategy of proof is still mostly the same, however, as the hypergroup structure is not needed. More importantly, the interplay between $\mathcal{H}$ and the convolution product still persists. Notation and first results are collected in section \ref{section-CO-notation} whereas as the new results concerning the Hausdorff--Young theorems for $\mathcal{H}$ and $\mathcal{H}^{-1}$ appear in section \ref{section.HY}.
\medskip

The second topic pertains to Paley--Wiener theorem for $\mathcal{H}$. The classical Paley--Wiener theorem for the Fourier transform describes the image of the space $C_c^\infty(\R^n)$ of compactly supported smooth functions in $\R^n$ as a space of entire functions in $\C^n$ of a specific exponential growth rate. Similar statements hold for various classes of Lie groups and symmetric spaces, as described succintly in several books by S. Helgason, for example. In the closely related framework of Jacobi analysis (in which the Jacobi transform replaces the spherical transform on a Riemannian symmetric space of rank one), a Paley--Wiener theorem of the aforementioned type was obtained in \cite{Tom-Mogens}, for example. As already summarized in theorem \ref{thm.PW} there are analogous results for the Cherednik--Opdam transform associated with arbitrary rank root systems in $\R^n$. There are comparable results in the framework of Dunkl theory, but the literature is not nearly as coherent and some details still seem to be missing.

Recently a `dual' question has started to attract attention, namely, to describe those functions in $L^2$ whose Fourier transform, the Dunkl transform, or Helgason-Fourier transform are smooth and compactly supported. This amounts to a Paley--Wiener theorem for the inverse transform. For Riemannian symmetric spaces such results and closely related so-called \emph{real} Paley--Wiener theorems were obtained in \cite{Andersen-pacific} and \cite{Pasquale-pacific}, for the Jacobi transform in \cite{Andersen-JMAA}, for the one-dimensional Dunkl transform in \cite{Chettaoui-Trimeche} and \cite{Andersen-PW-Dunkl}, and recently for the Heckman--Opdam transform in $\R^n$ in \cite{Mejjaoli-Trimeche}. We complement the literature by proving comparable results for the one-dimensional (inverse) Cherednik--Opdam transform, following the approach in \cite{Andersen-PW-Dunkl} and \cite{Andersen-JMAA}. These matters are detailed in section \ref{section.PW}.

\section{Notation and first properties}\label{section-CO-notation}
We are concerned with harmonic analysis for the Cherednik--Opdam transform acting on functions on $\R$. As already indicated elsewhere this is a convenient extension of Jacobi analysis on $\R^+$. While Opdam gave a brief discussion of the rank one case in \cite{Opdam-acta}, we shall rely exclusively on the \cite{Anker-Ayadi-Sifi} since it provides a much more detailed investigation of the eigenfunctions, the convolution structure and aspects of harmonic analysis.

In the following we fix parameters $\alpha, \beta$ subject to the constraints $\alpha\geq\beta\geq-\frac{1}{2}$ and $\alpha>-\frac{1}{2}$. Let $\rho=\alpha+\beta+1$ and $\lambda\in\C$. The Opdam hypergeometric functions $\G_\lambda^{(\alpha,\beta)}$ on $\R$ are eigenfunctions $T^{(\alpha,\beta)}\G_\lambda^{(\alpha,\beta)}(x)=i\lambda\G_\lambda^{(\alpha,\beta)}(x)$ of the differential-difference operator
\[T^{(\alpha,\beta)}f(x)= f'(x)+\Bigl[(2\alpha+1)\coth x+(2\beta+1)\tanh x\Bigr]\frac{f(x)-f(-x)}{2}-\rho f(-x)\]
that are normalized such that $G_\lambda^{(\alpha,\beta)}(0)=1$. In the notation of Cherednik one would write $T^{(\alpha,\beta)}$ as
\[T(k_1,k_2)f(x)= f'(x)+\Bigl\{\frac{2k_1}{1+e^{-2x}}+\frac{4k_2}{1-e^{-4x}}\Bigr\}(f(x)-f(-x))-(k_1+2k_2)f(x),\]
with $\alpha=k_1+k_2-\frac{1}{2}$ and $\beta=k_2-\frac{1}{2}$. Here $k_1$ is the multiplicity of a simply positive root and $k_2$ the (possibly vanishing) multiplicity of a multiple of this root. By \cite[page~90]{Opdam-acta} or \cite[formula~1.2]{Anker-Ayadi-Sifi}, the eigenfunction $\G_\lambda$ is given by
\[
\G_\lambda^{(\alpha,\beta)}(t)=\varphi_\lambda^{(\alpha,\beta)}(t)-\frac{1}{\rho-i\lambda}\frac{\partial}{\partial x}\varphi_\lambda^{(\alpha,\beta)}(t) = \varphi_\lambda^{(\alpha,\beta)}(t)+\frac{\rho}{4(\alpha+1)}\sinh(2t)\varphi_\lambda^{(\alpha+1,\beta+1)}(t),\]
where $\varphi_\lambda^{(\alpha,\beta)}(x)={_2}F_1(\frac{\rho+i\lambda}{2},\frac{\rho-i\lambda}{2};\alpha+1;-\sinh^2x)$ is the classical Jacobi function.

\begin{definition} The associated Laplace operator is defined by $\mathcal{L}=(T^{(\alpha,\beta)})^2$.
\end{definition}
The operator $\mathcal{L}$ is essentially self-adjoint on $L^2(\R,d\mu)$, where the measure $d\mu$ is defined by $d\mu(x)=A_{\alpha,\beta}(|x|)\,dx$ and $A_{\alpha,\beta}(x)=(\sinh x)^{2\alpha+1}(\cosh x)^{2\beta+1}$. 

\begin{remark}
It is important to point out that the functions $\G_\lambda$ are \textbf{not} the so-called Jacobi--Dunkl functions considered, for example, in \cite{BenSalem-JD}, the latter function being defined as eigenfunctions for the operator $\widetilde{T}^{(\alpha,\beta)}$ defined by $\widetilde{T}^{(\alpha,\beta)}f(x)=T^{(\alpha,\beta)}f(x)+\rho f(-x)$.
\end{remark}

According to \cite[Theorem~3.2]{Anker-Ayadi-Sifi} there exists a family of signed measures $\mu_{x,y}^{(\alpha,\beta)}$ such that the product formula
\[\G_\lambda^{(\alpha,\beta)}(x)\G_\lambda^{(\alpha,\beta)}(y)=\int_\R\G_\lambda^{(\alpha,\beta)}(z)\,d\mu_{x,y}^{(\alpha,\beta)}(z)\]
holds for all $x,y\in\R$ and $\lambda\in\C$, where
\[d\mu_{x,y}^{(\alpha,\beta)}(z)=\begin{cases} \mathcal{K}_{\alpha,\beta}(x,y,z)A_{\alpha,\beta}(\vert z\vert)dy&\text{if } xy\neq 0\\
d\delta_x(z)&\text{if } y=0\\
d\delta_y(z)&\text{if } x=0
\end{cases}\] 
and 
\begin{multline*}
\mathcal{K}_{\alpha,\beta}(x,y,z) = M_{\alpha,\beta}\vert\sinh x\cdot\sinh y\cdot\sinh z\vert^{-2\alpha}\int_0^\pi g(x,y,z,\chi)_+^{\alpha-\beta-1}\\
\times\Bigl[1-\sigma_{x,y,z}^\chi+\sigma_{x,z,y}^\chi+\sigma_{z,y,x}^\chi+\frac{\rho}{\beta+\frac{1}{2}}\coth x\cdot \coth y\cdot \coth z (\sin \chi)^2\Bigr]\times (\sin \chi)^{2\beta}\,d\chi
\end{multline*}
if $x,y,z\in\R\setminus\{0\}$ satisfy the triangular inequality $\vert\vert x\vert-\vert y\vert\vert < \vert z\vert < \vert x\vert+\vert y\vert$, and $\mathcal{K}_{\alpha,\beta}(x,y,z)=0$ otherwise. Here 
\[\sigma_{x,y,z}^\chi = \begin{cases}
\frac{\cosh x\cdot \cosh y-\cosh z\cdot\cos\chi}{\sinh x\cdot \sinh y}&\text{if } xy\neq 0\\
0&\text{if } xy=0\end{cases} \quad \text{ for } x,y,z\in\R, \chi\in[0,1]\]
and
$g(x,y,z,\chi)=1-\cosh^2x-\cosh^2y\cosh^2z+2\cosh x\cdot\cosh y\cdot\cosh z\cdot\cos\chi$. 

The product formula is used to obtain explicit estimates for the generalized translation operators
\[\tau_x^{(\alpha,\beta)}f(y)=\int_\R f(z)\,d\mu_{x,y}^{(\alpha,\beta)}(z),\]
as well as the convolution  product of suitable functions $f,g$ on $\R$ defined by
\[f\star g(x)=\int_\R\tau_x^{(\alpha,\beta)}f(-y)g(y)A_{\alpha,\beta}(\vert y\vert)\,dy.\]

Since the convolution of functions that are not even is permitted and produces a function on $\R$ that is not an even function either it is clear that the Jacobi transform is inadequate for further studies. To this end we need following transform:

\begin{definition}
Let $\alpha\geq\beta\geq-\frac{1}{2}$ with $\alpha>-\frac{1}{2}$. The \emph{Cherednik--Opdam transform} of $f\in C_c(\R)$ is defined by
\[\curlH{f}(\lambda)=\int_\R f(x)\G_\lambda^{(\alpha,\beta)}(-x)A_{\alpha,\beta}(\vert x\vert)\,dx\text{ for all } \lambda\in\C.\]
\end{definition}
The inverse transform is given as 
\[\mathcal{J}g(x)=\int_\R g(\lambda)\G_\lambda^{(\alpha,\beta)}(x)\left(1-\frac{\rho}{i\lambda}\right)\,\frac{d\lambda}{8\pi\vert\cfct_{\alpha,\beta}(\lambda)\vert^2},\]
where one should take note of the asymmetry of the formulae compared with the inversion formula for the Jacobi transform; the factor $``\bigl(1-\frac{\rho}{i\lambda}\bigr)$'' is not present in Jacobi analysis. The $\cfct$-function that appears is the same as in Jacobi analysis, however:
\[\cfct_{\alpha,\beta}(\lambda)=\frac{\Gamma(2\alpha+1)}{\Gamma(\alpha+\frac{1}{2})}\frac{\Gamma(i\lambda)}{\Gamma(\alpha-\beta+i\lambda)} \frac{\Gamma\left(\frac{\alpha-\beta+i\lambda}{2}\right)}{\Gamma\left(\frac{\rho+i\lambda}{2}\right)} = \frac{\Gamma(\alpha+1)2^{\rho-i\lambda}\Gamma(i\lambda)}{\Gamma\left(\frac{\rho+i\lambda}{2}\right)\Gamma\left(\frac{\alpha-\beta+1+i\lambda}{2}\right)}.\]

According to \cite[Lemma~4.1]{Anker-Ayadi-Sifi} there is a close relation between $\mathcal{H}$ and the Jacobi transform $\mathcal{F}_{\alpha,\beta}$, expressed by the identity
\[\curlH{f}(\lambda) = 2\mathcal{F}_{\alpha,\beta}(f_e)(\lambda)+2(\rho+i\lambda)\mathcal{F}_{\alpha,\beta}(Jf_o)(\lambda) \text{ for }\lambda\in\C \text{ and } f\in C_c(\R),\]
where $f_e$ is the even part of $f$, $f_o$ the odd part, and $Jf_o(x):=\int_{-\infty}^xf_o(t)\,dt$.

The corresponding Plancherel formula was established in \cite[Theorem~9.13(3)]{Opdam-acta}, to the effect that
\[\begin{split}
\int_\R\vert f(x)\vert^2 A_{\alpha,\beta}(\vert x\vert)\,dx &=\int_0^\infty \bigl(\vert\curlH{f}(\lambda)\vert^2+\vert\curlH{\check{f}}(\lambda)\vert^2\bigr) \frac{d\lambda}{16\pi\vert\cfct_{\alpha,\beta}(\lambda)\vert^2}\\
&=\int_\R \curlH{f}(\lambda)\overline{\curlH{\check{f}}(-\lambda)}\left(1-\frac{\rho}{i\lambda}\right)\frac{d\lambda}{8\pi\vert\cfct_{\alpha,\beta}(\lambda)\vert^2},
\end{split}\]
where $\check{f}(x):=f(-x)$. We also note that it follows from the defining identity $T^{(\alpha,\beta)}\G_\lambda(x)=i\lambda\G_\lambda(x)$ that $\mathcal{L}\G_\lambda(x)=-\lambda^2\G_\lambda(x)$, and therefore -- as will be seen in the course of the proof of lemma \ref{gen.inv} below -- 
\begin{equation}\label{eqn.eigen}
\mathcal{H}(\mathcal{L}f)(\lambda)=-\lambda^2\mathcal{H}f(x)\quad,\quad f\in C_c^\infty(\R).
\end{equation}
Note that -- contrary to Jacobi analysis -- there is no $\rho$-shift, as the $\rho$ has already been included in the definition of $T^{(\alpha,\beta)}$. More generally, we record the following useful identity.

\begin{lemma}\label{gen.inv}
Let $f$ be a smooth function on $\R$ with the property that $\mathcal{L}^nf$ belongs to $L^2(\R,d\mu)$ for every $n\in\mathbb{N}_0$. Then
\[\int_\R|\mathcal{L}^nf(x)|^2\,d\mu(x) = \int_\R |\lambda|^{4n}|\mathcal{H}f(\lambda)|^2\,d\nu(\lambda)\text{ for every } n\in\mathbb{N}_0.\]
\end{lemma}
\begin{proof}
Choose $f$ as in the hypothesis and let $g\in C_c^\infty(\R)$. Since $\mathcal{L}$ is essentially self-adjoint on $L^2(\R,d\mu)$, it follows from the Plancherel theorem for $\mathcal{H}$ that
\[\begin{split}
\langle\curlH{(\mathcal{L}f)},\curlH{g}\rangle_{2,\nu} &= \langle \mathcal{L}f,g\rangle_{2,\mu} = \langle \curlH{F},\curlH{(\mathcal{L}g)}\rangle_{2,\nu}\\
&=\langle\curlH{f},-(\cdot)^2\curlH{g}\rangle_{2,\nu} = \langle -(\cdot)^2\curlH{f},\curlH{g}\rangle_{2,\nu},
\end{split}\]
that is, $\mathcal{H}(\mathcal{L}f)(\lambda)=-\lambda^2\curlH{f}(\lambda)$ for $\nu$-almost every $\lambda\in\R$. Therefore $\mathcal{H}(\mathcal{L}^nf)(\lambda)=(-1)^n\lambda^{2n}\curlH{f}(\lambda)$ for $\nu$-almost every $\lambda\in\R$ and every $n\in\mathbb{N}_0$. Since $\mathcal{L}^nf$ belongs to $L^2(\R,d\mu)$ for every $n$, the Plancherel theorem implies that
\[\bigl\|\mathcal{L}^nf\bigr\|^2_{2,\mu} = \bigl\|\mathcal{H}(\mathcal{L}^nf)\bigr\|^2_{2,\nu}=\int_\R|\lambda|^{4n}|\curlH{f}(\lambda)|^2\,d\nu.\]
\end{proof}

It is known from  \cite[Proposition~4.4]{Anker-Ayadi-Sifi} that $\tau_x^{(\alpha,\beta)}\G_\lambda^{(\alpha,\beta)}(y)=\G_\lambda^{(\alpha,\beta)}(x)\G_\lambda^{(\alpha,\beta)}(y)$ and $\curlH{\tau_x^{(\alpha,\beta)}f}(\lambda)=\G_\lambda^{(\alpha,\beta)}(x)\curlH{f}(\lambda)$ for $f\in C_c(\R)$, in addition to which $\curlH{(f\star g)}(\lambda)=\curlH{f}(\lambda)\cdot\curlH{g}(\lambda)$, cf. \cite[Proposition~4.9]{Anker-Ayadi-Sifi}.

\begin{prop}
Assume $1\leq p,q,r\leq\infty$ satisfy $\frac{1}{p}+\frac{1}{q}-1=\frac{1}{r}$. For every $f\in L^p(\R,d\mu)$ and $g\in L^q(\R,d\mu)$ the convolution product $f\star g$ belongs to $L^r(\R,d\mu)$ and $\|f\star g\|_{r,\mu}\leq C\|f\|_{p,\mu}\|g\|_{q,\mu}$, where $d\mu(x)=A_{\alpha,\beta}(x)dx$.
\end{prop}

\begin{lemma}\label{lemma-G-estimate}
\begin{enumerate}[label=(\roman*)]
\item The function $\G_0$ is strictly positive and 
\[\G_0(x)\lesssim\begin{cases} (1+x)e^{-\rho x}&\text{if }x\geq 0\\ e^{\rho x}&\text{if }x\leq 0\end{cases}\]
\item For every $\lambda\in\R$ and $x\in\R$ it holds that $|\G_\lambda(x)|\leq\G_0(x)$.
\end{enumerate}
\end{lemma}
\begin{proof} See \cite[Lemma~5.2]{Anker-Ayadi-Sifi}.
\end{proof}

The following theorem of Paley--Wiener type was established by Opdam and Cherednik for arbitrary root systems in $\R^n$, see also \cite{Schapira-contributions}.

\begin{theorem}\label{thm.PW}
\begin{enumerate}[label=(\roman*)]
\item	The transform $\mathcal{H}$ is a topological isomorphism from $C_c^\infty(\R)$ onto $\mathrm{PW}(\C)$, where $\mathrm{PW}(\C)$ is the space of entire functions $h$ on $\C$ such that
	\[\exists R\geq 0\,\forall N_0\in\mathbb{N}:\, \sup_{\lambda\in\C}(1+|\lambda|)^N e^{-R|\Re\,\lambda|}h(\lambda)<\infty.\]
	\item For every $R>0$, $\mathcal{H}$ is a topological isomorphism of $C^\infty_R(\R)$ onto $\mathrm{PW}_R(\C)$. Here $C_R^\infty(\R)$ denotes the subspace of functions $f\in C_c^\infty(\R)$ that are supported in $[-R,R]$ and $\mathrm{PW}_R(\C)$ denotes the space of entire functions $h$ for which 
	\[\forall N\in\mathbb{N}_0:\sup_{\lambda\in\C}(1+|\lambda|)^Ne^{-\gamma(-\Re\,\lambda)}h(\lambda)<\infty\]
	where $\gamma(\lambda)=\sup_{\lambda\in[-R,R]}\langle\lambda,x\rangle$.
\end{enumerate}
\end{theorem}
We shall later obtain a Paley--Wiener theorem for the inverse transform $\mathcal{I}$. Since the map $\mathcal{H}$ is not self-dual, this is not automatic.

\section{Around the Hausdorff--Young inequality}\label{section.HY}

\begin{lemma}\label{lemma-HY1}
Let $\alpha\geq\beta\geq-\frac{1}{2}$ with $\alpha\neq-\frac{1}{2}$ and let $p\in[1,2)$, $q=\frac{p}{p-1}$. Define
\[\Omega_p=\Bigl\{\lambda=\xi+i\eta\in\C\,\bigl|\, \vert\eta\vert<\bigl(\tfrac{2}{p}-1\bigr)\rho\Bigr\}.\]
\begin{enumerate}
\item If $\lambda\in \Omega_p$ then $G_\lambda\in L^q(\R,d\mu)$.
\item If $f\in L^p(\R,d\mu)$, then $\mathcal{H}f(\lambda)$ is well-defined and holomorphic in $\lambda\in \Omega_p$, and $\vert\mathcal{H}f(\lambda)\vert\leq\|f\|_{p,\mu}\|G_\lambda\|_{q,d\nu}$ for all $\lambda\in \Omega_p$.
\item There exists a constant $c_p<\infty$ such that $\|\mathcal{H}f\|_{q,\nu}\leq c_p\|f\|_{p,\mu}$ for every $f\in L^p(\R,d\mu)$.
\end{enumerate}
\end{lemma}
\begin{proof}
The first two statements follow as in the case of Jacobi analysis, cf. \cite[Lemma~3.1]{Tom-Mogens}, and the third statement is established by interpolation between the estimates $\|\mathcal{H}f\|_{2,\nu}\lesssim\|f\|_{2,\mu}$ (which is the Plancherel theorem) and the uniform estimate $\|\mathcal{H}f\|_{\infty,\nu}\lesssim\|f\|_{1,\mu}$ that follows from lemma \ref{lemma-G-estimate}.
\end{proof}

\begin{remark}
It was observed in \cite{Eguchi-Kumahara-Lp} that the aforementioned Hausdorff--Young inequality can be improved considerably, by using the fact that the Fourier transform on $G/K$ -- and in our case the Cherednik--Opdam transform -- is well-defined for $\lambda\in\Omega_p$. The following extension of the Hausdorff--Young inequality is analogous to \cite[Lemma~5.3]{Narayanan-Pasquale-Pusti} but we shall not need it in later parts of the present paper. One simply notices that the non-symmetric Plancherel density $(1-\rho/i\lambda)|\cfct(\lambda)|^{-2}$ decays like $|\cfct(\lambda)|^{-2}$ for large $|\lambda|$.
\begin{lemma}\label{lemma-HY2}\begin{enumerate}
\item Let $f\in L^p(\R,d\mu)$ for some $p\in(1,2)$ and $\eta\in\Omega_p$. Then there is a positive constant $C_{\eta,p}$ such that
\[\Bigl(\int_\R |\curlH{f}(\lambda+\eta)|^q\,d\nu(\lambda)\Bigr)^{1/q}\leq C_{\eta,p}\|f\|_{p,\mu}\quad\text{for all } f\in L^p(\R,d\mu).\]
\item $\sup_{\lambda\in\R}|\curlH{f}(\lambda+\eta)|\leq C_{\eta,p}\|f\|_{p,\mu}$.
\end{enumerate}
\end{lemma}

We recently obtained an extension of the Hausdorff--Young inequality and several versions of the classical Hardy--Littlewood inequalities for the Heckman--Opdam transform acting on $W$-invariant functions in $\R^n$. These were based on the Hausdorff--Young inequalities obtained in \cite{Narayanan-Pasquale-Pusti}. By using the Hausdorff--Young inequalities in lemma \ref{lemma-HY1}(iii) and lemma \ref{lemma-HY2}, one can generalize in a straightforward manner  \cite[Theorem~3.6; 3.9]{Johansen-HL} to the present context of Cherednik--Opdam analysis on $\R$. We leave the details to the interested reader. 
\end{remark}
If $f\in L^1(\R,d\mu)\cap L^p(\R,d\mu)$ for some $p\in(1,2)$, the Hausdorff--Young inequality implies that $\curlH{f}$ is well-defined and belongs to $L^q(\R,d\nu)$. The transform $\mathcal{H}$ therefore extends to a continuous linear map $\mathscr{H}_p:L^p(\R,d\mu)\to L^q(\R,d\nu)$ that coincides with $\mathcal{H}$ on $L^1(\R,d\mu)\cap L^p(\R,d\mu)$. In particular, by the Plancherel theorem, we may identify $\mathcal{H}$ with $\mathscr{H}_2$. We employ a similar notational convention to define the map $\mathscr{I}_p:L^p(\R,d\nu)\to L^q(\R,d\mu)$ as the extension of $\mathcal{I}=\mathcal{H}^{-1}$.

\begin{prop}\label{prop-injective}
	Let $f\in L^p(\R,d\mu)$ for some $p\in [1,2]$ and assume $\mathscr{H}_pf=0$ $\nu$-almost everywhere. Then $f=0$ $\mu$-almost everywhere, that is, $\mathscr{H}_p$ is injective on $L^p(\R,d\mu)$
\end{prop}
\begin{proof}
	The argument is a slight modification of the proof for the Jacobi transform, resp. for the Heckman--Opdam transform, cf. \cite[Theorem~5.4]{Narayanan-Pasquale-Pusti}.
	\medskip
	
	Associate to a fixed function $g\in C_c^\infty(\R)$ the linear functionals 
	\[T_g(h)=\int_\R h(x)\overline{g(x)}\,d\mu(x),\quad \widehat{T}_g(h)=\int_\R \curlH{h}(\lambda)\overline{\curlH{\check{g}}(\lambda)}\,d\nu(\lambda),\quad h\in L^p(\R,d\mu).\]
	These coincide on the space $L^1(\R,d\mu)\cap L^2(\R,d\mu)$ which is dense in $L^p(\R,d\mu)$. We claim that $T_g$ and $\widehat{T}_g$ are continuous, from which it follows that they coincide on $L^p(\R,d\mu)$. In particular, if $f\in L^p(\R,d\mu)$ and $\curlH{f}=0$, then $\widehat{T}_g(h)=T_g(h)=0$, and $\langle f,g\rangle_{L^2(\R,d\mu)}=0$ for all $g\in C_c^\infty(\R)$. But then $f=0$.
	
	The continuity of $T_g$ follows from the obervation that $|T_g(h)|\leq \|h\|_{p,\mu}\|g\|_{q,\mu}$ where $\frac{1}{p}+\frac{1}{q}=1$. Additionally
$|\widehat{T}_g(h)|\leq\|\curlH{g}\|_{p,\nu}\|\curlH{h}\|_{q,\nu} \leq C_p\|\curlH{g}\|_{p,\nu} \|h\|_{p,\nu}$, 	so $\widehat{T}_g$ is also continuous. For the first inequality we used the Paley--Wiener estimate in theorem \ref{thm.PW}. This completes the proof.
\end{proof}
\begin{prop}
\begin{enumerate}[label=(\roman*)]
\item If $f\in L^{p_1}(\R,d\mu)\cap L^{p_2}(\R,d\mu)$, then $\mathscr{H}_{p_1}f=\mathscr{H}_{p_2}f$ $\nu$-almost everywhere on $\R$.
\item If $h\in L^{p_1}(\R,d\nu)\cap L^{p_2}(\R,d\nu)$, then $\mathscr{I}_{p_1}h=\mathscr{I}_{p_2}h$ $\mu$-almost everywhere on $\R$.
\end{enumerate}
\end{prop}
\begin{proof}
\begin{enumerate}[label=(\roman*)]
\item Choose a sequence $\{g_n\}_{n=1}^\infty$ of simple functions on $\R$ such that 
\[\lim_{n\to\infty}\|f-g_n\|_{p_1,\mu}=\lim_{n\to\infty}\|f-g_n\|_{p_2,\mu}=0.\]
Each function $\curlH{g_n}$ belongs to $L^{p_1^\prime}(\R,d\nu)\cap L^{p_2^\prime}(\R,d\nu)$ by the Hausdorff--Young inequality, and
\[\lim_{n\to\infty}\|\mathscr{H}_{p_1}f-\curlH{g_n}\|_{p_1^\prime,\nu}=\lim_{n\to\infty}\|\mathscr{H}_{p_2}f-\curlH{g_n}\|_{p_2^\prime,\nu}=0.\]
One can therefore extract subsequences $\{\curlH{g_{n_k}}\}_{k=1}^\infty$ and $\{\curlH{g_{n_l}}\}_{l=1}^\infty$ of $\{\curlH{g_n}\}_{n=1}^\infty$ such that $\curlH{g_{n_k}}\to\mathscr{H}_{p_1}f$ and $\curlH{g_{n_l}}\to\mathscr{H}_{p_2}f$ $\nu$-almost everywhere on $\R$, from which it follows that $\mathscr{H}_{p_1}f=\mathscr{H}_{p_2}f$ $\nu$-almost everywhere on $\R$ as claimed.
\item Since $\mathscr{I}_{q_j}h\in L^{q_j^\prime}(\R,d\mu)$ for $j=1,2$, the claim follows from the injectivity of $\mathscr{H}_p$ for $p\in(1,2]$ and (i).
\end{enumerate}
\end{proof}

\begin{lemma}\begin{enumerate}[label=(\alph*)]
		\item For $f,g\in L^2(\R,d\mu)$ and $h\in L^1(\R,d\mu)$ the convolution product $f\star g$ belongs to $C_0(\R)$ and 
		\[\int_\R (f\star g)(y)\check{h}(y)\,d\mu(y)=\int_\R \check{f}(y)(g\star h)(y)\,d\mu(y).\]
		\item For $f\in L^p(\R,d\mu)$, $p\in[1,2]$, and $h\in L^{p}(\R,d\nu)$ it holds that $\mathcal{I}(\mathscr{H}_pf\cdot h)=f\star(\mathscr{I}_p{h})$.
	\end{enumerate}
\end{lemma}
\begin{proof}
The first statement follows from an application of Fubini's theorem and commutativity of the convolution product.

As for the statement in (b), we observe that the inverse Cherednik--Opdam transform $\mathcal{I}(\mathscr{H}_pf\cdot h)$ is well-defined since $\|\mathscr{H}_pf\cdot h\|_{1,\nu}\leq \|\mathscr{H}_pf\|_{p^\prime}\|h\|_{p,\nu}<\infty$. Assume $f,h\in C_c(\R)$; then  $\mathcal{H}(f\star(\mathscr{I}_ph))=\mathcal{H}f\cdot \mathcal{H}(\mathscr{I}_ph) = \mathscr{H}_pf\cdot h$, and it follows from proposition \ref{prop-injective} that $\mathcal{I}(\mathscr{H}_pf\cdot h)=f\star(\mathscr{I}_ph)$ in this case.

For $f\in C_c(\R)$ and $h\in L^p(\R,d\nu)$ there is a sequence $\{h_n\}$ of functions $h_n\in C_c(\R)$ such that $\|h-h_n\|_{\infty,\nu}\to 0$ as $n\infty$. We have just seen that $\mathcal{I}(\mathscr{H}_pf\cdot h_n)=f\star(\mathscr{I}_ph_n)$, so several applications of the H\"older inequality, together with Young's inequality, lead to the estimate

\begin{multline*}
\|\mathcal{I}(\mathscr{H}_pf\cdot h)-f\star (\mathscr{I}_ph)\|_{\infty,\mu} \\
\begin{split} 
&= \|\mathcal{I}(\mathscr{H}_pf(h-h_n))+\mathcal{I}(\mathscr{H}_p h_n) - f\star (\mathscr{I}_p(h-h_n)) - f\star (\mathscr{I}_ph_n)\|_{\infty,\mu}\\
&\leq\|\mathcal{I}(\mathscr{H}_pf(h-h_n))-f\star(\mathscr{I}_p(h-h_n))\|_{\infty,\mu}\text{ since } \mathcal{I}(\mathscr{H}_pf h_n)=f\star(\mathscr{I}_ph_n)\\
&\leq \|\mathscr{H}_pf (h-h_n)\|_{1,\nu}+\|f\star(\mathscr{I}_ph)-f\star(\mathscr{I}_ph_n)\|_{\infty,\mu}\\
&\leq \|\mathscr{H}_pf\|_{p^\prime,\nu}\|h-h_n\|_{p,\nu} + \|f\|_{p,\mu}\|\mathscr{I}_p(h-h_n)\|_{p^\prime,\mu}\\
&\leq 2\|f\|_{p,\mu}\|h-h_n\|_{p,\nu}
\end{split}
\end{multline*}
For the general case where $f\in L^p(\R,d\mu)$ one chooses a sequence $\{f_n\}$ in $C_c(\R)$ such that $\|f-f_n\|_{p,\mu}\to 0$ as $n\to\infty$.	
\end{proof}

\begin{lemma}\label{lemma-help}
	\begin{enumerate}[label=(\roman*)]
	\item For every compact neighborhood $C$ of a given $\lambda\in\R$ there exists a net $\{f_i\}_n$ of functions $f_n\in C_c(\R)$ such that $\mathcal{H}f_i\to \mathbf{1}_C$ uniformly on compact subsets and $\|\mathcal{H}(f_i\star\overline{f}_i)-\mathbf{1}_C\|_{1,\nu}\to 0$ as $n\to\infty$.
			\item Assume $f\in L^p(\R,d\mu)$ for some $p\in[1,2]$ has the property that $\mathscr{H}_pf$ belongs to $L^2(\R,d\nu)$. Then $f\in L^2(\R,d\mu)$, and $f=\mathcal{I}(\mathscr{H}_pf)$ $\mu$-almost everywhere.
			\item Assume $h\in L^p(\R,d\nu)$ for some $p\in[1,2]$ has the property that $\mathscr{I}_ph$ belongs to $L^2(\R,d\mu)$. Then $h\in L^2(\R,d\nu)$ and $h=\mathcal{H}(\mathscr{I}_ph)$ $\nu$-almost everywhere.
	\end{enumerate}
\end{lemma}
\begin{remark}
Statement (i) appears as \cite[Lemma~2.6]{HY-hyper} for hypergroups, including a reference with the proof. Although the convolution $\star$ presently does not define a hypergroup structure, it is still sensible to view the $G_\lambda$ as analogues of the characters. In the special case of Jacobi analysis this analogy is indeed correct.

We shall therefore merely outline the necessary adjustments to the proof of (i), (ii), and (iii) as they appeared in \cite{HY-hyper} that are required to take into account the minor differences.
\end{remark}
\begin{proof}[Proof of (i)]
Let $\lambda\in\R$ and $\epsilon>0$ be fixed, and let $C\subset\R$ be a compact subset. Since $\lambda\mapsto\G_\lambda(x)$ is continuous for fixed $x$, an equicontinuity argument establishes the existence of an open neighborhood $U$ of $\lambda$ in $\R$ such that $0<\nu(U)<\infty$ and 
\[U\subset \bigl\{\varphi\in\R\,:\, |\G_\lambda(x)-\G_\varphi(x)|<\epsilon/2\text{ for all } x\in C\bigr\}.\]
Defining $h=\nu(U)^{-1}\mathbf{1}_U$ (which belongs to $L^1(\R,d\nu)$) it follows for $x\in C$ that
\[\begin{split}
|\mathcal{I}h(x)-\G_\lambda(x)| &=\Bigl|\int_\R\G_\varphi(x)h(\varphi)\,d\nu(\varphi)-\G_\lambda(x)\Bigr|\\
&=\Bigl|\frac{1}{\nu(U)}\int_U\G_\varphi(x)\,d\nu(\varphi)-\frac{1}{\nu(U)}\int_U\G_\lambda(x)\,d\nu(\varphi)\Bigr|\\
&\leq \frac{1}{\nu(U)}\int_U|\G_\varphi(x)-\G_\lambda(x)|\,d\nu(\varphi)<\frac{\epsilon}{2}
\end{split}\]
by  the construction of the set $U$. Setting $k=\nu(U)^{-1/2}\mathbf{1}_U$, there exists a function $f\in C_c(\R)$ such that $\|\mathcal{H}f-k\|_{2,\nu}<\epsilon/4$. Since $\|k\|_{2,\nu}=1$, it suffices to consider the case there $\|\mathcal{H}f\|_{2,\nu}=1=\|f\|_{2,\mu}$. Because $f$ belongs to $C_c(\R)$, it holds that $\mathcal{H}(f\star\overline{f})=(\mathcal{H}f)\cdot (\overline{\mathcal{H}f})=|\mathcal{H}f|^2$, and therefore
\[\begin{split}
\|\mathcal{H}(f\star\overline{f})-h\|_{1,\nu} &=\bigl\| |\mathcal{H}f|^2-h\bigr\|_{1,\nu}\\
&\leq \int_\R|\mathcal{H}f(\varphi)-k(\varphi)|\cdot |\mathcal{H}(\overline{f})(\varphi)+k(\varphi)|\,d\nu(\varphi)\\
&\leq \|\mathcal{H}f-k\|_{2,\nu}\bigl(\|\mathcal{H}f\|_{2,\nu}+\|k\|_{2,\nu}\bigr)\\
&\leq 2\|\mathcal{H}f-k\|_{2,\nu} < \frac{\epsilon}{2}.
\end{split}\]
For arbitrary $x\in C$ it now holds that $|f\star\overline{f}(x)-\mathcal{I}h(x)|\leq\|\mathcal{H}(f\star\overline{f})-h\|_{1,\nu}<\epsilon/2$, whence $|f\star\overline{f}(x)-\G_\lambda(x)|\leq |f\star\overline{f}(x)-\mathcal{I}h(x)|+|\mathcal{I}h(x)-\G_\lambda(x)|<\epsilon$ uniformly in $x\in C$.
\end{proof}
\begin{proof}[Proof of (ii)]
Choose a net $\{h_i\}$ of functions in $C_c(\R)$ such that $h_i\star f\to f$ in $L^p(\R,d\mu)$ and such that $\mathcal{H}f_i$ converges uniformly on compact subsets of $\R$ to the identity. Moreover choose a net $(f_j)$ of functions in $C_c(\R)$ such that $\|f_j-f\|_{p,\mu}\to 0$. It follows from $h_i\star f$ belonging to $L^p(\R,d\mu)\cap C_0(\R)$ that $h_i\star f$ is in $L^2(\R,d\mu)$, implying that $\mathscr{H}_p(h_i\star f)=\mathcal{H}(h_i\star f)\in L^2(\R,d\nu)$. By the Hausdorff--Young inequality it holds that $\mathcal{H}(h_i\star f)$ also belongs to $L^{p^\prime}(\R,d\nu)$. Now
\begin{multline*}
\|\mathcal{H}(h_i\star f)-\mathscr{H}_p(h_i)\mathscr{H}_pf\|_{p^\prime,\nu} \\
\begin{split}
& \leq \|\mathcal{H}(h_i\star f)-\mathscr{H}_p(h_i\star f_j)\|_{p^\prime,\nu} + \|\mathscr{H}_p(h_i\star f_j)-\mathscr{H}_p(h_i)\mathscr{H}_p(f)\|_{p^\prime,\nu}\\
&=\|\mathcal{H}h_i\cdot(\mathcal{H}f-\mathscr{H}_pf_j)\|_{p^\prime,\nu} + \|\mathscr{H}_p(h_i)\cdot\mathscr{H}_p(f_j-f)\|_{p^\prime,\nu}\\
&\longrightarrow 0 
\end{split}\end{multline*}
from which we conclude that $\mathcal{H}(h_i\star f)=\mathscr{H}_p(f_i)\cdot\mathscr{H}_p(f)$ $\nu$-almost everywhere. Therefore, according to the Plancherel theorem, 
\[\|k_i\star f-\mathcal{I}(\mathscr{H}_p(f))\|_{2,\mu}=\|\mathcal{H}(k_i\star f)-\mathscr{H}_pf\|_{2,\nu} = \|\mathcal{H}h_i\cdot\mathscr{H}_pf\|_{2,\nu}.\]
By using (i) we can choose the compact set $C$ in the beginning of the proof so that
\begin{multline*}
\int_C |(\mathcal{H}h_i-1)(\lambda)\mathscr{H}_pf(\lambda)|^2\,d\nu(\lambda) + \int_{\complement C}|(\mathcal{H}h_i-1)(\lambda)\mathscr{H}_pf(\lambda)|^2\,d\nu(\lambda) \\ < \int_C |(\mathcal{H}h_i-1)(\lambda)\mathscr{H}_pf(\lambda)|^2\,d\nu(\lambda) + \frac{\epsilon}{2} \longrightarrow \frac{\epsilon}{2}.\end{multline*}
In particular $\|k_i\star f-\mathcal{I}(\mathscr{H}_pf)\|_{2,\mu}\to 0$, implying that $f=\mathcal{I}(\mathscr{H}_pf)$ $\mu$-almost everywhere.\end{proof}
\begin{proof}[Outline of proof for (iii)]
Let $C$ be a compact neighborhood of a fixed $\lambda\in\R$ and use (i) to produce a net $\{f_i\}$ in $C_c(\R)$ such that  $\|\mathcal{H}(f_i\star\overline{f}_i)-\mathbf{1}_C\|_{1,\nu}\to 0$. For a given $i$, let $h=\mathcal{H}(f_i\star\overline{f}_i)$; then $h$ belongs to $L^1(\R,d\nu)\cap L^2(\R,d\nu)$ by Plancherel and moreover to $C_c(\R)$. At this point one can repeat the argument in \cite{HY-hyper}. We leave the details to the interested reader.
\end{proof}
The following result is the first main theorem and is analogous to \cite[Theorem~2.8]{HY-hyper}.
\begin{theorem} Assume $\alpha\geq\beta\geq -\frac{1}{2}$ with $\alpha>-\frac{1}{2}$, and let $p,r\in[1,2]$.
\begin{enumerate}[label=(\roman*)]
\item For $f\in L^p(\R,d\mu)$ such that $\mathscr{H}_pf\in L^r(\R,d\nu)$ it follows that $\mathscr{I}_r(\mathscr{H}_pf)=f$ in $L^p(\R,d\mu)$.
\item For $h\in L^p(\R,d\nu)$ such that $\mathscr{I}_p\in L^r(\R,d\mu)$ it follows that $\mathscr{H}_r(\mathscr{I}_ph)=h$ in $L^p(\R,d\nu)$. 
\end{enumerate}
\end{theorem}
\begin{proof}
In the first case $\mathscr{H}_pf$ belongs to $L^{p^\prime}(\R,d\nu)\cap L^r(\R,d\nu)\subset L^2(\R,d\nu)$, implying that $\mathscr{H}_pf=\mathcal{H}f$. But then $f=\mathcal{I}(\mathcal{H}f)=\mathscr{I}_{p^\prime}(\mathscr{H}_pf)$ according to proposition \ref{lemma-help}.

The second statement is proved analogously.
\end{proof}

\section{Paley--Wiener theorems for the inverse Cherednik--Opdam transform}\label{section.PW}
The present section establishes several Paley--Wiener-type results for the inverse transform $\mathcal{I}$. We shall follow \cite[Section~4]{Andersen-JMAA} closely, although some minor differences occur due to the Plancherel measure $d\nu$ being non-symmetric and the $\G_\lambda$ being more complicated.

\begin{definition}
The support of $g\in L^2(\R,d\mu)$ is the smallest closed set, on the complement of which $g$ vanishes almost everywhere. The $\emph{radius}$ of the support of $g$ is defined by
\[R_g:=\sup_{\lambda\in\mathrm{supp}\,g}|\lambda|.\]
\end{definition}

The real Paley--Wiener theorems hinge on the following description of $R_g$, cf. \cite[Lemma~5]{Andersen-JMAA} and the references quoted, as well as \cite[Theorem~4.1]{Andersen-PW-Dunkl} and \cite[Theorem~4.1]{Chettaoui-Trimeche} in the Dunkl-case.  The proof goes through without any change.

\begin{lemma}\label{lemma-Rg}
Assume $g\in L^2(\R,d\nu)$ has the property that $|\cdot|^{2n}g\in L^2(\R,d\nu)$ for all $n\in\mathbb{N}_0$. Then
\[R_g=\lim_{n\to\infty}\Bigl(\int_\R|\lambda|^{4n}|g(\lambda)|^2\,d\nu(\lambda)|\Bigr)^{\frac{1}{4n}}.\]
\end{lemma}
\begin{proof}
The conclusion being trivial for $g=0$ (in which case $R_g=0$), we assume without loss of generality that $\|g\|_{2,\nu}\neq 0$.

First consider the case where $g$ has compact support with $R_g>0$. Since 
\[\Bigl(\int_\R|\lambda|^{4n}|g(\lambda)|^2\,d\nu(\lambda)\Bigr)^{\frac{1}{4n}}\leq R_g\Bigl(\int_{\vert\lambda\vert\leq R_g}\lambda|^{4n}|g(\lambda)|^2\,d\nu(\lambda)\Bigr)^{\frac{1}{4n}},\]
it follows that
\[\limsup_{n\to\infty}\Bigl(\int_\R|\lambda|^{4n}|g(\lambda)|^2\,d\nu(\lambda)\Bigr)^{\frac{1}{4n}}\leq R_g\limsup_{n\to\infty} \Bigl(\int_{\vert\lambda\vert\leq R_g}\lambda|^{4n}|g(\lambda)|^2\,d\nu(\lambda)\Bigr)^{\frac{1}{4n}} = R_g.\]

On the other hand
\[\int_{R_g-\epsilon\leq|\lambda|\leq R_g}|g(\lambda)|^2\,d\nu(\lambda)>0\text{ for every }\epsilon>0,\]
from which it follows that
\[\begin{split}
\liminf_{n\to\infty}\Bigl(\int_\R|\lambda|^{4n}|g(\lambda)|^2\,d\nu(\lambda)\Bigr)^{\frac{1}{4n}} & \geq \liminf_{n\to\infty}\Bigl(\int_{R_g-\epsilon\leq|\lambda|\leq R_g}|\lambda|^{4n}|g(\lambda)|^2\,d\nu(\lambda)\Bigr)^{\frac{1}{4n}}\\
&\leq (R_g-\epsilon)\liminf_{n\to\infty}\Bigl(\int_{R_g-\epsilon\leq|\lambda|\leq R_g}|g(\lambda)|^2\,d\nu(\lambda)\Bigr)^{\frac{1}{4n}} \\ & = R_g-\epsilon,
\end{split}\]
Therefore $\displaystyle \lim_{n\to\infty}\Bigl(\int_\R|\lambda|^{4n}|g(\lambda)|^2\,d\nu(\lambda)\Bigr)^{\frac{1}{4n}} =R_g$, which was the desired conclusion.
\medskip

Now assume $g$ has unbounded support, and note that
\[\int_{|\lambda|\geq N}|g(\lambda)|^2\,d\nu(\lambda)>0\text{ for every } N>0.\]
But then
\[\begin{split}
\liminf_{n\to\infty}\Bigl(\int_\R|\lambda|^{4n}|g(\lambda)|^2\,d\nu(\lambda)\Bigr)^{\frac{1}{4n}}&\geq \liminf_{n\to\infty}\Bigl(\int_{|\lambda|\geq N}|\lambda|^{4n}|g(\lambda)|^2\,d\nu(\lambda)\Bigr)^{\frac{1}{4n}}\\
&\leq N\cdot \liminf_{n\to\infty}\Bigl(\int_{|\lambda|\geq N}|g(\lambda)|^2\,d\nu(\lambda)\Bigr)^{\frac{1}{4n}} = N
\end{split}\]
for every $N>0$, implying that
\[\liminf_{n\to\infty}\Bigl(\int_\R|\lambda|^{4n}|g(\lambda)|^2\,d\nu(\lambda)\Bigr)^{\frac{1}{4n}} = +\infty.\]
\end{proof}

\begin{definition}
The real $L^2$-based Paley--Wiener space $\mathrm{PW}^2(\R)$ is the space of smooth functions $f$ on $\R$ satisfying
\begin{enumerate}[label=(\roman*)]
\item $\mathcal{L}^nf\in L^2(\R,d\mu)$ for every $n\in\mathbb{N}_0$,
\item $R_f^\mathcal{L}:=\liminf_{n\to\infty}\|\mathcal{L}^nf\|_{2,\mu}^{1/2n}<\infty$.
\end{enumerate}
In addition, define $\mathrm{PW}^2_R(\R)=\{f\in\mathrm{PW}^2(\R)\,:\,R_f^\mathcal{L}=R\}$ for $R\geq 0$, $L^2_c(\R)$ the space of functions in $L^2(\R,d\nu)$ with compact essential support, and $L^2_R(\R)=\{g\in L^2_c(\R)\,:\,R_g=g\}$.
\end{definition}

\begin{theorem}\label{thm-PW1}
The inverse Cherednik--Opdam transform $\mathcal{I}$ is bijective from $L_c^2(\R)$ onto $\mathrm{PW}^2(\R)$, and from $L_R^2(\R)$ onto $\mathrm{PW}^2_R(\R)$.
\end{theorem}
\begin{proof}
If $g\in L_R^2(\R)$, then $|\cdot|^n g$ belongs to $L^1(\R,d\nu)\cap L^2(\R,d\nu)$ for every $n\in\mathbb{N}_0$ and $\mathcal{I}g$ is therefore a smooth function on $\R$. Moreover $\mathcal{L}^n(\mathcal{I}g)=\mathcal{I}((-1)^n(\cdot)^{2n}g)\in L^2(\R,d\mu)$ for all $n\in\mathbb{N}_0$ by definition of $\mathcal{I}$ and the identity $\mathcal{L}\G_\lambda=-\lambda^2\G_\lambda$. We conclude from lemma \ref{lemma-help} that 
\[\liminf_{n\to\infty}\Bigl(\int_\R|\mathcal{L}^n(\mathcal{I}g)(x)|^2\,d\mu(x)\Bigr)^{\frac{1}{4n}} = \liminf_{n\to\infty}\Bigl(\int_\R|\lambda|^{4n}|g(\lambda)|^2\,d\nu(\lambda)\Bigr)^{\frac{1}{4n}} = R,\]
which means that $\mathcal{I}g$ belongs to $\mathrm{PW}^2_R(\R)$.
\medskip

If on the other hand $f\in\mathrm{PW}^2_R(\R)$, then $\mathcal{H}(\mathcal{L}^nf)(\lambda)=(-1)^n\lambda^{2n}\curlH{f}(\lambda)\in L^2(\R,d\nu)$ for all $n\in\mathbb{N}_0$. Therefore
\[\lim_{n\to\infty}\Bigl(\int_\R|\lambda|^{4n}|\mathcal{H}f(\lambda)|^2\,d\nu(\lambda)\Bigr)^{\frac{1}{4n}} = \lim_{n\to\infty}\Bigl(\int_\R|\mathcal{L}^nf(x)|^2\,d\mu(x)\Bigr)^{\frac{1}{4n}}=R,\]
that is, $\curlH{f}$ has compact support and $R_{\curlH{f}}=R$.
\end{proof}

As in \cite[Corollary~8]{Andersen-JMAA} one obtains the following

\begin{corollary}
Let $g$ be a measurable function on $\R$. Then $|\cdot|^ng$ belongs to $L^2(\R,d\nu)$ for all $n\in\mathbb{N}_0$ if and only if $\mathcal{L}^n\mathcal{I}g$ belongs to $L^2(\R,d\mu)$ for all $n\in\mathbb{N}_0$.
\end{corollary}

\begin{definition}
The Paley--Wiener space $\mathrm{PW}(\R)$ consists of all functions $f\in C^\infty(\R)$ satisfying
\begin{enumerate}
\item $(1+|\cdot|)^m\mathcal{L}^n f\in L^2(\R,d\mu)$ for all $m,n\in\mathbb{N}_0$,
\item $R_f^\mathcal{L}f:=\lim_{n\to\infty}\|\mathcal{L}^nf\|_{2,\mu}^{1/2n}<\infty$.
\end{enumerate}
In addition, let $\mathrm{PW}_R(\R):=\{f\in\mathrm{PW}(R)\,:\,R_f^\mathcal{L}=R\}$ for $R\geq 0$ and $C_R^\infty(\R)=\{g\in C_c^\infty(\R)\,:\,R_g=g\}$.
\end{definition}

\begin{theorem}
The inverse Cherednik--Opdam transform $\mathcal{I}$ is a bijection of $C_c^\infty(\R)$ onto $\mathrm{PW}(\R)$, and $\mathcal{I}$ maps $C_R^\infty(\R)$ onto $\mathrm{PW}_R(\R)$ for every $R\geq 0$.
\end{theorem}
\begin{proof}
First note that $\mathcal{I}g$ belongs to $\mathrm{PW}^2_R(\R)$ whenever $g\in C_R^\infty(\R)$, according to theorem \ref{thm-PW1}. Since $\mathcal{L}^n(\mathcal{I}g)=\mathcal{I}((-1)^n(\cdot)^{2n}g)\in L^2(\R,d\mu)$ for every $n\in\mathbb{N}_0$, it remains to show that $\mathcal{I}g$ satisfies the polynomial decay condition whenever $g\in C_c^\infty(\R)$. Recall that
\[\mathcal{I}g(x)=\int_\R\G_\lambda(x)g(\lambda)\Bigl(1-\frac{\rho}{i\lambda}\Bigr)\frac{d\lambda}{8\pi|\cfct_{\alpha,\beta}(\lambda)|^2},\]
where $|\cfct_{\alpha,\beta}|^{-2}$ is the density in the symmetric Plancherel measure that appears in Jacobi analysis (and in the the proof of \cite[Theorem~11]{Andersen-JMAA}), so one must prove that
\begin{equation}\label{eqn.boundedX}
x\mapsto (1+|x|)^m\int_\R\lambda^{2n}\G_\lambda(x)g(\lambda)\Bigl(1-\frac{\rho}{i\lambda}\Bigr)\frac{d\lambda}{8\pi|\cfct_{\alpha,\beta}(\lambda)|^2}
\end{equation}
belongs to $L^2(\R,d\mu)$ for every $m,n\in\mathbb{N}_0$. While the argument given for this fact in the proof of \cite[Theorem~11]{Andersen-JMAA} does not generalize immediately to the present setting, the basic idea is still sound. The issue is that for the Jacobi case one can use that the Jacobi function $\varphi_\lambda^{(\alpha,\beta)}$ is even in $\lambda$ and allows for a Harish-Chandra series expansion adapted to the $\cfct$-function. One then proceeds with an argument involving Gangolli estimates for the coefficients in the Harish-Chandra expansion, much like in the proof for the Paley--Wiener theorem for the Jacobi transform. Due to the asymmetry in $d\nu$, such an expansion will necessarily be more complicated, albeit it is still available (as detailed in 
\cite[Section~3]{Schapira-contributions}).

Let us instead explain how to derive the required decay estimate directly from repeated use of \cite[Theorem~11]{Andersen-JMAA}: In \eqref{eqn.boundedX} we can write
\begin{multline*}
\G_\lambda(x)g(\lambda)\Bigl(1-\frac{\rho}{i\lambda}\Bigr) = \Bigl(1-\frac{\rho}{i\lambda}\Bigr)(g_e(\lambda)+g_o(\lambda))\varphi_\lambda^{(\alpha,\beta)}(x) \\
+ \frac{\rho}{4(\alpha+1)}\Bigl(1-\frac{\rho}{i\lambda}\Bigr)\sinh(2x)(g_e(\lambda)+g_o(\lambda))\varphi_\lambda^{(\alpha+1,\beta+1)}(x)\end{multline*}
where $g_e$ and $g_o$ denote the even and odd part of $g$, respectively.
This function is integrated against the symmetric Plancherel density $|\cfct_{\alpha,\beta}(\lambda)|^{-2}$, which decays according to the asymptotic estimate $|\cfct_{\alpha,\beta}(\lambda)|^{-2}\asymp|\lambda|^{2\alpha+1}$. From \cite[Theorem~11]{Andersen-JMAA} it follows that
\[x\mapsto (1+|x|)^m\int_\R\lambda^{2m}g_e(\lambda)\varphi_\lambda^{(\alpha,\beta)}(x)\,|\cfct_{\alpha,\beta}(\lambda)|^{-2}\,d\lambda\]
lies in $L^2(\R,d\mu)$ for every $m,n\geq 0$. Since $\lambda\mapsto (1-\rho/i\lambda)|\cfct_{\alpha,\beta}(\lambda)|^{-2}$ is locally integrable near $0$ and $|1-\rho/i\lambda|\lesssim 1$ for $|\lambda|\gg 1$, it follows from $\lambda\mapsto g_0(\lambda)\varphi_\lambda(x)$ being odd that the function
\[x\mapsto (1+|x|)^m\int_\R\Bigl(1-\frac{\rho}{i\lambda}\Bigr)\lambda^{2m}(g_e(\lambda)+g_0(\lambda))\varphi_\lambda^{(\alpha,\beta)}(x)\,|\cfct_{\alpha,\beta}(\lambda)|^{-2}\,d\lambda\]
also belongs to $L^2(\R,d\mu)$ for every $m,n\geq 0$. In order to deal with the remaining contribution to the integrand in \eqref{eqn.boundedX} we write
\begin{multline*}
\sinh(2x)g(\lambda)\varphi_\lambda^{(\alpha+1,\beta+1)}(x)|\cfct_{\alpha,\beta}(\lambda)|^{-2}\\ =\sinh(2x)g(\lambda)\frac{|\cfct_{\alpha,\beta}(\lambda)|^{-2}}{|\cfct_{\alpha+1,\beta+1}|^{-2}}\varphi_\lambda^{(\alpha+1,\beta+1)}(x)|\cfct_{\alpha+1,\beta+1}(\lambda)|^{-2},\end{multline*}
the point being that 
\[\biggl| \frac{|\cfct_{\alpha,\beta}(\lambda)|^{-2}}{|\cfct_{\alpha+1,\beta+1}(\lambda)|^{-2}}\biggr|\asymp \frac{|\lambda|^{2\alpha+1}}{|\lambda|^{2(\alpha+1)+1}}=|\lambda|^{-2}\]
introduces additional \emph{decay} in $|\lambda|$. An application of \cite[Theorem~11]{Andersen-JMAA} with Jacobi parameters $(\alpha+1,\beta+1)$ and associated weighted measure $d\mu_{\alpha+1,\beta+1}(x)=(\sinh x)^{2(\alpha+1)+1}(\cosh x)^{2(\beta+1)}\,dx$ guarantees that
\[\int_\R\Bigl|(1+|x|)^m\int_\R g_e(\lambda)|\lambda|^{2n}\varphi_{\alpha+1,\beta+1}(x)\,\frac{d\lambda}{|\cfct_{\alpha+1,\beta+1}(\lambda)|^2}\Bigr|^2\, A_{\alpha+1,\beta+1}(|x|)\,dx<\infty\]
which implies, in particular, that 
\[\int_\R\Bigl|(1+|x|)^m\sinh(2x)\int_\R \lambda^{2n} g_e(\lambda)\varphi_\lambda^{(\alpha+1,\beta+1)}(x)\,\frac{d\lambda}{|\cfct_{\alpha,\beta}(\lambda)|^2}\Bigr|^2\,A_{\alpha,\beta}(|x|)\,dx<\infty\]
As explained above, the factor $(1-\rho/i\lambda)$ can be included in the estimates without further issues, which concludes the proof that $\mathcal{I}g$ indeed belongs to $\mathrm{PW}(R)$.

Finally observe that a function $f\in \mathrm{PW}_R(\R)\subset\mathrm{PW}^2_\R(\R)$ has polynomial decay, so that $f$ $\curlH{f}$ is smooth. As $|\G_\lambda(x)|\lesssim (1+|x|)e^{-\rho|x|}$, it follows from theorem \ref{thm-PW1} that $\curlH{f}$ is compactly supported with $R_{\curlH{f}}=R$. This completes the proof.
\end{proof}

\providecommand{\bysame}{\leavevmode\hbox to3em{\hrulefill}\thinspace}
\providecommand{\MR}{\relax\ifhmode\unskip\space\fi MR }
\providecommand{\MRhref}[2]{%
  \href{http://www.ams.org/mathscinet-getitem?mr=#1}{#2}
}
\providecommand{\href}[2]{#2}

\end{document}